\documentclass[11pt]{amsart}
\usepackage{latexsym}
\usepackage{rotating}
\usepackage{amsfonts}
\usepackage{amssymb,amscd}
\usepackage{amsmath}
\usepackage{graphics, color}
\usepackage[margin=3.75 cm]{geometry}

\def\r{\mathcal R}
\def\R{\mathbb R}

\def\C{\mathbb C}

\def\F{\mathbb F}
\def\H{\mathbb H}

\def\Pa{{\rm P}}
\def\z{{\bf z}}
\def\w{{\bf w}}
\def \ve {{\bf v}}
\def\hom{{\rm Hom}}

\def\PGL{\rm  {PGL}}
\def\PO {\mathrm{ PO}}

\def\V{\mathrm V}

\def\F{{\rm F}}
\def\X{\mathbb X}
\def\A{\mathbb A}

\def \rh{{\bf H}_{\R}}
\def \ch{{\bf H}_{\C}}
\def \h{{\bf H}_{\H}}
\def\R{\mathbb R}

\newcommand{\SL}{\mathrm{SL}}
\newcommand{\SU}{\mathrm{SU}}
\newcommand{\GL}{\mathrm{GL}}
\newcommand{\Sp}{\mathrm{Sp}}

\def\P{\mathbb P}
\def\M{{\mathrm M}}

\def \a {{\bf a}}
\def \r  {{\bf r}}
\def \x {{\bf x}}

\def  \bchi {\mathfrak D}
\def  \dchi {\mathfrak {D_{L}}}

\newtheorem{theorem}{Theorem}[section]
\newtheorem{corollary}[theorem]{Corollary}
\newtheorem{lemma}[theorem]{Lemma}
\newtheorem{prop}[theorem]{Proposition}
\theoremstyle{definition}
\newtheorem{definition}[theorem]{Definition}

\newtheorem*{ack}{Acknowledgement}
\theoremstyle{remark}
\newtheorem{remark}[theorem]{Remark}
\numberwithin{equation}{section}

\numberwithin{equation}{section}

\newcommand{\secref}[1]{Section~\ref{#1}}
\newcommand{\thmref}[1]{Theorem~\ref{#1}}
\newcommand{\lemref}[1]{Lemma~\ref{#1}}
\newcommand{\remref}[1]{Remark~\ref{#1}}
\newcommand{\propref}[1]{Proposition~\ref{#1}}
\newcommand{\corref}[1]{Corollary~\ref{#1}}

\begin{document}

\title[ Quaternionic Hyperbolic Fenchel-Nielsen Coordinates]{ Quaternionic Hyperbolic Fenchel-Nielsen Coordinates}
\author[Krishnendu Gongopadhyay  \and Sagar B. Kalane]{Krishnendu Gongopadhyay \and
 Sagar B. Kalane}
\address{Indian Institute of Science Education and Research (IISER) Mohali,
 Knowledge City,  Sector 81, S.A.S. Nagar 140306, Punjab, India}
\email{krishnendug@gmail.com, krishnendu@iisermohali.ac.in}
\address{Indian Institute of Science Education and Research (IISER) Mohali,
 Knowledge City,  Sector 81, S.A.S. Nagar 140306, Punjab, India}
\email{sagark327@gmail.com}
\date{March 5, 2018}
\begin{abstract}
Let $\Sp(2,1)$ be the isometry group of the quaternionic hyperbolic plane $\h^2$. An element $g$ in $\Sp(2,1)$ is \emph{hyperbolic} if it fixes exactly two points on the boundary of $\h^2$. 
We classify  pairs of hyperbolic elements in $\Sp(2,1)$ up to conjugation.  

A hyperbolic element of  $\Sp(2,1)$ is called \emph{loxodromic} if it has no real eigenvalue.  We show that the set of $\Sp(2,1)$ conjugation orbits of irreducible loxodromic pairs is a $(\C\P^1)^4$-bundle over a topological space that is locally a semi-analytic subspace of  $\R^{13}$.  We use the above classification to show that  conjugation orbits  of `geometric' representations of a closed surface group (of genus $g \geq 2$) into $\Sp(2,1)$ can be determined by a system of  $42g-42$ real parameters. 

Further, we consider the groups $\Sp(1,1)$ and $\GL(2, \H)$. These groups also act by the orientation-preserving isometries of the four and five dimensional real hyperbolic spaces respectively. We classify conjugation orbits of pairs of hyperbolic elements in these groups. These classifications determine conjugation orbits of `geometric' surface group representations into these groups.  
\end{abstract}
\maketitle

\def\cdprime{$''$} \def\Dbar{\leavevmode\lower.6ex\hbox to 0pt{\hskip-.23ex
  \accent"16\hss}D}

\section{Introduction}
Let $\Gamma$ be a finitely generated discrete group and let $G$ be a connected Lie group. A problem of potential interest is to understand the geometry and topology of the deformation space, or the conjugation orbit space  $\bchi(\Gamma, G)=\hom(\Gamma, G)/G$.  Here $G$ acts on $\hom(\Gamma, G)$  by conjugation via inner automorphisms. Particularly interesting spaces appear when $\Gamma=\pi_1(\Sigma_g)$ and $G$ is the isometry group of a rank one symmetric space of non-compact type. Here $\Sigma_g$ denotes a closed connected orientable surface of genus $g \geq 2$ and $\pi_1(\Sigma_g)$, or simply $\pi_1$, denotes the fundamental group of $\Sigma_g$.

  When $G=\SL(2, \R)$, the space $\bchi(\pi_1(\Sigma_g), \SL(2, \R))$ contains the classical Teichm\"uller space as one of its components, and has been studied widely in the literature, though not completely understood even today, for example see the survey \cite{goldmcg} or the recent work \cite{mw}. When $G=\SL(2, \C)$, the deformation space contains the so called quasi-Fuchsian space that is related to Thurston's program on three manifolds and Kleinian groups, see Marden \cite{marden}.  Let $\SU(2,1)$ be the isometry group of the two dimensional complex hyperbolic space $\ch^2$. Attempts have been made to understand the discrete, faithful and geometrically finite components in the deformation space $\bchi(\pi_1, \SU(2,1))$. Though some progress has been made in this direction, but the deformation space is far from being well-understood, see the surveys \cite{pp2}, \cite{will3}.

A common problem in all the above directions is to look for coordinate systems on the respective deformation spaces. In other words,  one 
 wants to parametrize special classes of  representations (e.g. discrete, faithful, geometrically finite representations)  in the deformation space.  In the case of $G=\SL(2, \R)$, this lead to the Fenchel-Nielsen coordinates on the Teichm\"uller space, see \cite{sw}. This was generalized by Kourounitis \cite{kou} and Tan \cite{tan} for $G=\SL(2, \C)$.  Parker and Platis \cite{pp} obtained parametrization  of the `complex hyperbolic quasi-Fuchsian' representations when $G=\SU(2,1)$. Will  obtained trace coordinates for such representations in \cite{will2}. 
Gongopadhyay and Parsad \cite{gp2} obtained partial generalization of the work of Parker and Platis for $G=\SU(3,1)$, and have given parametrization of  generic classes of three dimensional complex hyperbolic quasi-Fuchsian representations.

A starting point to such a parametrization  is the classification of pairs of hyperbolic isometries. This has been used to parametrize the so called $(0,3)$ subgroups and then the (totally hyperbolic) representations are parametrized by gluing of these $(0,3)$ subgroups using `twist coordinates'. We recall that a  $(0,3)$ subgroup represents the fundamental group of a three-holed sphere, where the generators and their product correspond to the three boundary curves.  It corresponds  to the representations $\rho$ of $\F_2=\langle x, y \rangle$ into $G$ such that both $\rho(x)$, $\rho(y)$, and their product are hyperbolic elements in $G$. 

  When $G=\SL(2, \R)$, the starting point to parametrize a $(0,3)$ subgroup is the classical works of Fricke and Vogt that classify  a pair $(A, B)$ in $\SL(2, \R)$, up to conjugacy,  by the traces of $A$, $B$ and $AB$, see the survey \cite{gold2} for an exposition in this direction. In the two-dimensional complex hyperbolic space, generalization of the result of Fricke and Vogt follows from a work of Lawton who has classified pairs of elements in $\SL(3, \C)$ up to conjugacy in \cite{law}. In \cite{will2}, Will applied Lawton's work to obtain the classification of hyperbolic pairs in $\SU(2,1)$ using trace coordinates, also see \cite{par}. The work of Will has been extended to the case $G=\SU(3,1)$ in \cite{gl} where the authors obtained a set of trace coordinates to determine any representation whose conjugation orbit is closed.  Mostly, linear algebraic methods from classical invariant theory have been used in these works. However, in the case when both $A$ and $B$ are hyperbolic isometries,  the situation is more geometric, and different methods using  complex hyperbolic geometry may be taken. Working from this viewpoint, 
 Parker and Platis \cite{pp}, Falbel \cite{falbel} and Cunha and Gusevskii \cite{cugu1}, took different approaches to classify conjugation orbits of  pairs of hyperbolic elements in $\SU(2,1)$.  \\
 
 In this paper, we investigate the case $G=\Sp(2,1)$. The group $\Sp(2,1)$ acts on the two dimensional quaternionic hyperbolic space $\h^2$ by isometries.  Elements of $\Sp(2,1)$ are  
$3 \times 3$ matrices over the division ring $\H$ of quaternions satisfying a quaternionic Hermitian form of signature $(2,1)$.  So, this case has an additional complexity that arises due to the non-commutativity of the quaternions. For example, conjugacy invariants like traces or, eigenvalues of a quaternionic  matrix, are not well-defined.  So, special care needs to be taken in order to associate conjugacy invariants to hyperbolic  element in $\Sp(2,1)$. Since real numbers commute with any non-zero quaternionic numbers, based on the nature of eigenvalues we may divide the hyperbolic elements in $\Sp(2,1)$ into four types.   A hyperbolic element $A$ in $\Sp(2,1)$ will be called   \emph{loxodromic} if the similarity classes of eigenvalues of $A$  have representatives in the set $\H \setminus \R$.  If all the eigenvalues of $A$ are real numbers, it is called \emph{strictly hyperbolic}.  The other two  types will not be mentioned here, and the reader is referred to \remref{rtype} in \secref{type} for a brief note on them.

\medskip The main result of this paper is a classification of pairs of hyperbolic elements in $\Sp(2,1)$. A pair $(A, B)$ in $\Sp(2,1)$ is called \emph{totally loxodromic} if both $A$ and $B$ are loxodromic. It is called \emph{strictly hyperbolic} if both $A$ and $B$ are strictly hyperbolic. Recall that a pair $(A, B)$ of $\Sp(2,1)$ is \emph{irreducible} if it neither fixes a point, nor preserves a proper  totally geodesic subspace of $\h^2$. The subset of $\bchi(\F_2, \Sp(2,1))$ consisting of  irreducible representations $\rho$ such that both $\rho(x)$ and $\rho(y)$ are loxodromic, is denoted by $\dchi_o(\F_2, \Sp(2,1))$.

\medskip  Let $(A, B)$ be a pair in $\Sp(2,1)$ such that both $A$ and $B$ are hyperbolic elements. Embedding $\Sp(2, 1)$ into $\GL(6, \C)$: $g \mapsto g_{\C}$, gives us certain conjugacy invariants which are co-efficients of the characteristic polynomial of $g_{\C}$, that we call \emph{real trace} of $g$. When $g$ is loxodromic, the real trace lies on an open subspace ${\rm D}_2$ of $\R^3$. 

For a hyperbolic element $A$ in $\Sp(2,1)$, its fixed points are denoted by $a_A$ and $r_A$.  Let $\M(2)$ denote the configuration space of $\Sp(2,1)$-congruence classes of ordered quadruples of distinct points on $\partial \h^2$. There is a point $p=(a_A, r_A, a_B, r_B)$ associated to $(A, B)$  in the space $\M(2)$.  There are numerical invariants like the quaternionic cross ratios and the angular invariants associated to this point $p$. These invariants are defined as follows. 

\medskip Let $A$ and $B$ be two loxodromic elements in $\Sp(2,1)$.   The fixed points of $A$ and $B$ on $\H\P^2 \setminus (\h^2 \cup \partial \h^2)$ are denoted by $x_A$ and $x_B$ respectively.  Let $\a_A$, $\r_A$, $\x_A$ denote the lifts of $a_A$, $r_A$, $x_A$. For the pair $(A, B)$ we define the following invariants.

\subsubsection*{The cross ratios of $(A, B)$} 
$$\X_1(A, B)=\X(a_A, r_A, a_B, r_B), ~~\X_2(A, B)=\X(a_A, r_B, a_B, r_A),$$
$$\X_3(A, B)=\X(r_A, r_B, a_B, a_A).$$

\subsubsection*{The angular invariants of $(A, B)$} 
$$\A_1(A, B)=\A(a_A, r_A, a_B), ~~ \A_2(A, B)=\A(a_A, r_A, r_B), ~~\A_3(A, B)=\A(r_A, a_B, r_B).$$

\medskip Recently, Cao has proved in \cite{cao} that these invariants  completely determined the point $p$. Cao has also obtained a topological description of the space $\M(2)$ in \cite{cao}. The space $\M(2)$ is locally parametrized by seven real parameters consisting of a point on the four (real) dimensional `cross ratio variety' and three real angular invariants.

\medskip However, the real traces of $A$ and $B$, along with the cross ratios and angular invariants of $p$, are not enough to classify the pair $(A, B)$ up to conjugacy by $\Sp(2,1)$.  A main challenge in the quaternionic set up is to look for the `missing' invariants to classify a pair. In this paper, we associate certain spatial invariants, called  the \emph{projective points} of an isometry, to determine the $\Sp(2,1)$ conjugation orbit of $(A, B)$. These invariants arise naturally from an understanding of the eigenvalue classes of an isometry, and is the key innovation in this paper.  It is not needed to associate projective point to a real eigenvalue of a hyperbolic element. So, to describe the concept, we may assume that $A$ is loxodromic. Then to a non-real eigenvalue representative $\lambda$, consider the eigenspace of $\lambda$. We identify it with $\H$, that is again identified with $\C^2$. An `eigenset' of $\lambda$ can be identified with $\C$ and there is a decomposition of the eigenspace into eigensets. The set of $\lambda$-eigensets can be identified with the set of complex lines in $\C^2$, that is the one-dimensional complex projective space $\C \P^1$. Each point on this $\C \P^1$ represents an eigenset of the eigenvalue $\lambda$, and is called a projective point of the eigenvalue class of $\lambda$.  A choice of projective points, along with `attracting' and `repelling' fixed points,  and real trace, determine a hyperbolic element in $\Sp(2,1)$. We prove the following. 

\begin{theorem}\label{plox}
Let $(A, B)$ be a totally loxodromic pair in $\Sp(2,1)$. Then $(A, B)$ is determined up to conjugation by the following parameters:  $tr_{\R}(A)$, $tr_{\R}(B)$, the similarity classes of cross ratios of $(A, B)$ or, a point on the four dimensional cross ratio variety, the three angular invariants of $(A, B)$, and the projective points $(p_1(A), p_2(A))$, $(p_1(B), p_2(B))$.  
\end{theorem}

Alternatively, associate to $(A, B)$ an ordered quadruple of its fixed points. This gives us the following version of the above theorem.
\begin{theorem}\label{cth1}
Let $(A, B)$ be a totally loxodromic pair in $\Sp(2,1)$. Then $(A, B)$ is uniquely determined up to conjugation by the following parameters:  $tr_{\R}(A)$, $tr_{\R}(B)$,  a point on ${\M}(2)$ corresponding to $(a_A, r_A, a_B, r_B)$,   and four projective points $(p_1(A), p_2(A))$, $(p_1(B), p_2(B))$.   
\end{theorem}

   We note that the above theorems are true for any pairs of loxodromics in $\Sp(2,1)$. However,  when the pair $(A, B)$ is irreducible and totally loxodromic,  the degrees of freedom of the parameters add up to $21$ (for each real traces $3$ contributing $2 \times 3=6$, for the point on the cross ratio variety $4$, for three angular invariants $3$ and for the four projective points $8=2 \times (4 \hbox{ projective points})$), which is the same as the dimension of the group $\Sp(2,1)$. When the pair $(A, B)$ is reducible, the parameter system will not be minimal and the degrees of freedom will be lesser than $21$. 

\medskip  
As noted above,  there are four types of hyperbolic elements in $\Sp(2,1)$. Depending on the types of $A$ and $B$, there are ten types of mixed hyperbolic pairs. Except the totally loxodromic pairs, the degrees of freedom of the parameter systems of the other types, are always  strictly lesser than $21$, even if the pair is irreducible.  A reason of this lesser degrees of freedom is the presence of real eigenvalues that commute with every other quaternionic numbers. Accordingly, we would not require projective points for the real eigenvalues present in the mixed pair.  For each of these cases, the parameter system can be  seen using similar arguments used in the proof of the above theorem.  We  only mention the following two cases. 

 \begin{corollary}\label{cormix}
Let $(A, B)$ be a  mixed hyperbolic pair with $A$ is loxodromic and $B$ is strictly hyperbolic in $\Sp(2,1)$. Then $(A, B)$ is uniquely determined up to conjugation by the parameters:  $tr_{\R}(A)$, $tr_{\R}(B)$, a point on the cross ratio variety, the three angular invariants and  projective points $p_1(A)$, $p_2(A)$. 
\end{corollary}

Noting that for a strictly hyperbolic element $A$ in $\Sp(2,1)$, $tr_{\R}(A)$ belongs to an one real parameter family,  the degrees of freedom of the associated parameters for an irreducible pair $(A, B)$ in the above corollary is  $15$, and  in the following case it is $9$. 

\begin{corollary}\label{corsh}
Let $(A, B)$ be a strictly hyperbolic pair in $\Sp(2,1)$. Then $(A, B)$ is uniquely determined up to conjugation by the parameters:  $tr_{\R}(A)$, $tr_{\R}(B)$, a point on the cross ratio variety and  three angular invariants. 
\end{corollary}

\medskip  As an application of \thmref{plox}, we have a parametric description of the subset $\dchi_o(\F_2, \Sp(2,1))$ of the deformation space consisting of irreducible totally loxodromic representations. 
\begin{corollary}\label{cord1}
The space $\dchi_o(\F_2, \Sp(2,1))$ is parametrized by a  $(\C \P^1)^4$ bundle over the topological space ${\rm D}_2 \times  {\rm D}_2 \times {\mathcal M}(2)$
where,  for $ G=27(c-a)-9ab+2a^3$,  $H=3(b-3)-a^2$,
 $${\rm D}_2=\{(a, b, c ) \in \R^3 \ | \ G^2 + 4H^3>0, ~|2a+c| \neq |2b+2| \},$$
and ${\mathcal M}(2)$ is the orbit space of the natural $S_4$ action on the configuration space of $\Sp(2,1)$-congruence classes of ordered quadruples of distinct points on $\partial \h^2$. 
\end{corollary}

\medskip  Another application of \thmref{plox} gives us local parametrization of generic elements in   $\bchi(\pi_1(\Sigma_g), \Sp(2,1))$. Let $\Sigma_g$ be a closed, connected, orientable surface of genus $g \geq 2$. Let $\pi_1(\Sigma_g)$ be the  fundamental group of $\Sigma_g$. Specify a curve system $\mathcal C$ of $3g-3$ closed curves $\gamma_j$ on $\Sigma_g$. The complement of such curve system decomposes the surface into $2g-2$ three-holed spheres $\Pa_i$.  Let $\rho: \pi_1(\Sigma_g) \to \Sp(2,1)$ be a discrete, faithful representation such that the image of each $\gamma_j$ is loxodromic. For each $i$, the fundamental group of $\Pa_i$ gives a representation $\varphi_i$  in $\dchi(\F_2, \Sp(2,1))$,  induced by $\rho$, such that the image of  $\varphi_i$ is a $(0, 3)$ subgroup of $\Sp(2,1)$. If each of these representations $\varphi_i$ is irreducible, we call the representation $\rho$ as \emph{geometric}. With this terminology, we have the following. 
\begin{theorem}\label{mnth2}
Let $\Sigma_g$ be a closed surface of genus $g$ with a curve system $\mathcal C=\{\gamma_j\}$, $j=1, \ldots, 3g-3$. Let $\rho: \pi_1(\Sigma_g) \to \Sp(2,1)$ be a geometric  representation of the surface group $\pi_1(\Sigma_g)$. Then we need $42g-42$ real parameters to determine $\rho$ in the deformation space $\bchi(\pi_1(\Sigma_g), \Sp(2,1))/\Sp(2,1)$.
\end{theorem}

The parameters in the above theorem may be thought of as `Fenchel-Nielsen coordinates' in this set up. Locally this gives us the degrees of freedom that a geometric representation can move. But the complete structure of this space is still not clear to us. It would be interesting to obtain an embedding of the geometric representations into a topological space.

\medskip \medskip Now, we briefly mention the case when the pair of hyperbolic elements $(A, B)$ is reducible. In this case, the parameters obtained in \thmref{plox} or the subsequent corollaries, would not be minimal. As in the complex case,  e.g. see \cite[Section 4.1]{gl},  it is possible to further reduce the degrees of freedom.  If a hyperbolic pair $(A, B)$ is reducible, it follows from \cite[Proposition 4.2]{cg} that the group $G$ generated by $A$, $B$,  is  a product of the form $\mathcal A \times \mathcal B$, where several possibilities of $\mathcal A$ and $\mathcal B$ are listed in \cite[p.77]{cg}. When the $G$-invariant totally geodesic subspace is a copy of the real or complex hyperbolic space, parameters to determine them can be obtained easily from the classical Fenchel-Nielsen coordinates, and by the work of Parker and Platis \cite{pp}.  Essentially,  conjugacy classification of hyperbolic pairs in $\Sp(1,1)$ remains the only case in order to have an understanding of the reducible pairs. For this reason, instead of addressing the reducible cases in detail, we have focused only on the irreducible hyperbolic pairs in $\Sp(1,1)$.  The rest of the cases are left to the interested reader.  

Let $\M(1)$ denote the configuration space of $\Sp(1,1)$-congruence classes of ordered quadruples of points on $\partial \h^1$.  We prove the following. 
\begin{theorem}\label{plox1}
Let $(A, B)$ be an irreducible totally loxodromic pair in $\Sp(1,1)$. Then $(A, B)$ is determined up to conjugation by the following parameters:  $tr_{\R}(A)$, $tr_{\R}(B)$,  a point on ${\M}(1)$ corresponding to $(a_A, r_A, a_B, r_B)$,  and two projective points $p_1(A)$, $p_1(B)$. 
\end{theorem}
The above theorem shows that the degrees of freedom of the coordinates required for a irreducible pair add up to $10$ (for each real traces $2$ contributing $2 \times 2=4$, for the point on the cross ratio variety $2$ and for the two projective points $2+2$, thus totaling 4+2+4=10, which is the same as the dimension of the Lie group $\Sp(1,1)$.  

\medskip Further note that, irreducibility is not necessary in the above theorem. But if $(A, B)$ is reducible, the parameter system above is not minimal and the degrees of freedom can be reduced further.

\medskip  The following corollaries follow immediately from the above theorem. 
\begin{corollary}\label{plox2}
Let $A$ be a loxodromic and $B$ be a strictly hyperbolic element  in $\Sp(1,1)$. Let $(A, B)$ be irreducible.  Then $(A, B)$ is determined up to conjugation by the following parameters:  $tr_{\R}(A)$, $tr_{\R}(B)$,   a point on ${\M}(1)$   corresponding to $(a_A, r_A, a_B, r_B)$,  and a projective point $p_1(A)$. 
\end{corollary}

\begin{corollary}\label{plox3}
Let $A, ~B$ be   strictly hyperbolic elements in $\Sp(1,1)$. Let $(A, B)$ be irreducible. Then $(A, B)$ is determined up to conjugation by the following parameters:  $tr_{\R}(A)$, $tr_{\R}(B)$, and   a point on ${\M}(1)$  corresponding to $(a_A, r_A, a_B, r_B)$.  
\end{corollary}

\medskip The group $\Sp(1,1)$ may also be viewed as the isometry group of the real hyperbolic $4$-space $\rh^4$. The space $\dchi_o(\F_2, \Sp(1,1))$ is of importance to understand four dimensional real hyperbolic geometry as well.  In \cite{tand}, Tan et. al. proved a counterpart of the classical Delambre-Gauss formula for right-angled hexagons in the real hyperbolic $4$-space keeping this deformation space in mind.  In this paper we  show that $\dchi_o(\F_2, \Sp(1,1))$ is a $\C \P^1 \times \C \P^1$ bundle over a topological space that is locally embedded in $\R^6$.  As a consequence of this description, it follows that a geometric  representation in $\bchi(\pi_1(\Sigma_g), \Sp(1,1))$ is determined by 20g-20  real parameters.

\medskip We recall that the group $\PGL(2, \H)$ is isomorphic to $\PO(5,1)$ and acts on the real hyperbolic $5$-space $\rh^5$ by isometries. The group $\mathrm{PSp}(1,1)$ is isomorphic to a subgroup of $\PGL(2, \H)$ that preserves a copy of $\rh^4$ inside $\rh^5$.  Using the approaches taken in this paper, it is also possible to state similar results for pairs of hyperbolic elements in $\GL(2, \H)$. Algebraic and dynamical classification of elements in $\GL(2, \H)$ are  available in \cite{g9}, \cite{ps}. There are three types of hyperbolic elements in $\GL(2, \H)$ depending upon the number of `rotation angles'. Two of these types come from $\Sp(1,1)$, and the third type called \emph{2-rotatory hyperbolic} in \cite{g9} or \emph{3-simple loxodromic} in \cite{pabo}, does not have representatives in $\Sp(1,1)$.  We shall address pairs of this type briefly, and will determine them up to conjugacy in $\GL(2, \H)$. A description of the subset of the deformation space $\dchi_o(\F_2, \GL(2, \H))$ consisting of totally $3$-simple loxodromic representations will also be noted.

\medskip Here is a brief summary of the structure of the paper. We recall basic notions and results needed for our purpose in \secref{prel}.  The pivotal concept of  `projective points'  is explained in \secref{proj}. We determine hyperbolic elements of $\Sp(2,1)$ in \secref{hyp}. This section is another important ingredient in proving \thmref{plox}.  The set ${\rm D}_2$ is also obtained in this section. We prove \thmref{plox} and \corref{cord1} in \secref{palo}.     In \secref{fn}, we prove \thmref{mnth2}.   We prove \thmref{plox1} and classify   pairs of $3$-simple loxodromic elements of $\GL(2, \H)$ in \secref{fn1}.

\medskip Due to their importance from several viewpoints, the deformation spaces in \secref{fn1} should be looked at with greater details. In this paper, we have only noted down the straightforward consequences which are evident from our understanding of the $\Sp(2,1)$ case. We hope to see a more thorough investigation of these deformation spaces in the future. 

\medskip As it is clearly visible in all the statements above, the notion of projective points is crucial for the development in this paper. The lack of numerical invariants over the quaternions has been supplemented by the use of these spatial invariants. We believe that there should be a classification of hyperbolic pairs in $\Sp(2,1)$ using purely numerical invariants. We hope  that a counterpart of \cite[Theorem 4.8]{par} will be available for pairs  in $\Sp(2, 1)$. We expect that the real traces of $A$, $B$, and several of their compositions, should be sufficient to classify the pair itself. A starting point in this direction could be the work of {\Dbar}okovi{\'c} and {Smith} in \cite{dos},  where a set of  minimal  real trace coordinates is available for $\Sp(2)$ conjugation orbits in  $\M_2(\H)$. The invariant theoretic methods in \cite{dos} might be possible to extend to classify conjugacy classes in $\Sp(2,1)$, and also to classify the conjugation orbits of pairs of elements in $\Sp(2, 1)$.

\section{Preliminaries}\label{prel}

\subsection{The Quaternions}Let $\H$ denote the division ring of quaternions.  Any $ q\in {\H}$ can be uniquely written as  $q=r_{0}+r_{1}i+r_{2}j+r_{3}k$, where $r_{0},r_{1},r_{2},r_{3}\in \R$, and  $ i,j,k$ satisfy relations:  $i^{2}=j^{2}=k^{2}=ijk=-1$. The real number $r_0$ is called the real part of $q$ and we denote it by $\Re(q)$. The imaginary part of $q$ is $\Im(q)=r_{1}i+r_{2}j+r_{3}k$. The conjugate of $q$ is defined as $\overline {q}= r_0-r_1 i -r_2 j - r_3 k$.
The norm of $q$ is $|q|=\sqrt{r_0^{2}+r_1^{2}+r_2^{2}+r_3^{2}}$.
We identify the sub-algebra $\R + \R i$ with the standard complex plane $\C$.

Two quaternions ${q_1, q_2}$ are said to be \emph{similar} if there exists a non-zero quaternion $ {z}$ such that $ {q_2=z^{-1}q_1z}$ and we write it as $ {q_1\backsim q_2}$. It is easy to verify that $ {q_1 \backsim q_2}$ if and only if $ {\Re(q_1)=\Re(q_2)}$ and $|q_1|=|q_2|$. Thus the similarity class of a quaternion $q$ contains a pair of complex conjugates with absolute-value $|q|$ and real part equal to $\Re( q)$.  The multiplicative group $\H \setminus 0$ is denoted by $\H^{\ast}$.
\subsubsection*{Commuting quaternions} Two non-real quaternions commute if and only if their imaginary parts are scaled by a real number, see  \cite[Lemma 1.2.2]{cg} for a proof. Let $Z(q)$ denotes the centralizer of $q \in \H \setminus \R$. Then $Z(q)=\R + \R q$. For some non-zero $\alpha \in \H$,  $Z(q)=\alpha \C \alpha^{-1}$, where $\C=Z(re^{i \theta})$, $r=|q|$, $\Re (q) =r \cos \theta$. Given a quaternion $q\in \H \setminus \R $, we call $Z(q)$ the \emph{complex line} passing through $q$.
 Note that if $q \in \H \setminus \R$, then $Z(\overline{q}^{-1})={Z(q)}$.

\subsection{Matrices over quaternions} 
Let $\V$ be a right vector space over $\H$. Let $T$ be a right linear map on $\V$. After choosing a basis of $\V$, such a linear map can be represented by an $n \times n$ matrix $M_T$ over $\H$, where $n=\dim \V$. Thus, one may identify the right linear maps on $\V$ with $n \times n$ quaternionic matrices. Here the quaternionic matrices act on $\V$ from the left. The map $T$ is invertible if and only if $M_T$ is invertible.  The group of all invertible (right) linear maps on $\V$  is denoted by $\GL(n, \H)$. 

 Suppose $\lambda\in \H^{\ast}$ is a (right) eigenvalue of $T$, i.e. there exists $v \in \V$ such that $T(v)=v \lambda$. Observe that for $\mu \in \H^{\ast}$,
$$T(v \mu)=T(v) \mu=(v \lambda )\mu=(v \mu) \mu^{-1} \lambda \mu.$$ 
This shows that the eigenvalues of $T$ occur in similarity classes and if $v$ is a $\lambda$-eigenvector, then $v \mu \in v \H$ is a $\mu^{-1} \lambda \mu$-eigenvector.
 Thus the eigenvalues are no more conjugacy invariants of $T$, but the similarity classes of eigenvalues are conjugacy invariants.   Note that each similarity class of eigenvalues contains a unique pair of complex conjugate numbers. We shall choose one of these complex eigenvalues $re^{i \theta}$, $\theta \in [0, \pi]$,  to be the representative of its similarity class. Often we shall refer them as `eigenvalues', though it should be understood that our reference is towards their similarity classes. In places, where we need to distinguish between the similarity class and a representative, we shall denote the similarity class of an eigenvalue representative $\lambda$ by $[\lambda]$. 

For more information on quaternionic linear algebra, we refer to the book \cite{lr}. 
\subsection{Quaternionic Hyperbolic Space}
Let $\V=\H^{2,1}$ be the $3$ dimensional right vector over $\H$ equipped with the Hermitian form of signature $(2,1)$ given by 
$$\langle\z,\w\rangle=\w^{\ast}H_1\z=\bar w_3 z_1+\bar w_2 z_2+\bar w_1 z_3,$$
where $\ast$ denotes conjugate transpose. The matrix of the Hermitian form is given by
\begin{center}
$H_1=\left[ \begin{array}{cccc}
            0 & 0 & 1\\
           0 & 1 & 0 \\
 1 & 0 & 0\\
          \end{array}\right].$
\end{center}
We consider the following subspaces of $\H^{2,1}:$
$$\V_{-}=\{\z\in\H^{2,1}:\langle\z,\z \rangle<0\}, ~ \V_+=\{\z\in\H^{2,1}:\langle\z,\z \rangle>0\},$$
$$\V_{0}=\{\z-\{{\bf 0}\}\in\H^{2,1}:\langle\z,\z \rangle=0\}.$$
A vector $\z$ in $\H^{2,1}$ is called \emph{positive, negative}  or \emph{null}  depending on whether $\z$ belongs to $\V_+$,   $\V_-$ or  $\V_0$. There are two distinguished points in $\V_{0}$ which we denote by  $\bf{o}$ and $\bf\infty,$ given by
$$\bf{o}=\left[\begin{array}{c}
               0\\0\\  1\\
              \end{array}\right],
~~ \bf\infty=\left[\begin{array}{c}
               1\\0 \\ 0\\
              \end{array}\right].$$\\

Let $\P:\H^{2,1}-\{0\}\longrightarrow  \H \P^2$ be the right projection onto the quaternionic projective space. Image of a vector $\z$ will be denoted by $z$.  The quaternionic hyperbolic space $\h^2$ is defined to be $\P \V_{-}$. The ideal boundary of $\h^2$ is $\partial \h^2=\P \V_{0}$.  For a point $\z=\begin{bmatrix}z_1 & z_2 &  z_{3}\end{bmatrix}^T \in \V_- \cup \V_0$, projective coordinates are given by $(w_1, w_2)$, where $w_i=z_i z_{3}^{-1}$ for $i=1,2$. In projective coordinates we have 
$$\h^2=\{(w_1,w_2)\in \H^2 : \ 2\Re(w_1)+|w_2|^2<0\},$$
$$\partial\h^2-\infty=\{(w_1,w_2)\in \H^2 :2\Re(w_1)+|w_2|^2=0\}.$$

This is the Siegel domain model of $\h^2$. Similarly one can define the ball model by replacing $H_1$ with an equivalent Hermitian form given by the diagonal matrix:
$$H_2=\begin{bmatrix} -1 & 0 & 0 \\ 0 & 1& 0 \\ 0 & 0 & 1 \end{bmatrix}.$$
 We shall mostly use the Siegel domain model here. 

Given a point $z$ of $\h^2 \cup \partial \h^2 -\{\infty\} \subset\H \P^2$ we may lift $z=(z_1,z_2)$ to a point $\z$ in $\V_0 \cup \V_{-}$, given by 
 $$\z=\left[\begin{array}{c}
                z_1\\z_2\\1\\
               \end{array}\right].$$
$\z$ is called the \emph{standard lift} of $z$. 

For $z$ and $w$ in $\h^2$,  the Bergman metric is given by the distance function $\rho$:
$$\cosh^2\big(\frac{\rho(z,w)}{2}\big)=\frac{\langle\z,\w\rangle \langle\w,\z\rangle}{\langle\z,\z\rangle \langle\w,\w\rangle}.$$

\subsection{ Hyperbolic Isometries }   \label{ltr}
 Let ${\rm Sp}(2,1)$ be the isometry group of  the Hermitian form $H_1$. Each matrix $A$ in ${\rm Sp}(2,1)$ satisfies the relation $A^{-1}=
H_1^{-1}A^{\ast}H_1$, where $A^{\ast}$ is the conjugate transpose of $A$. The isometry group of  $\h^2$ is the projective unitary group ${\rm PSp}(2,1)={\rm Sp }(2,1)/\{\pm I\}$. However, we shall mostly deal with the linear group $\Sp(2,1)$.

\begin{definition} An isometry in $\Sp(2,1)$  is \emph{hyperbolic} if it fixes exactly two points on $\partial\h^2$. If the eigenvalues of a hyperbolic isometry $g$ are all real numbers, we call it \emph{ strictly hyperbolic}. If all the eigenvalues of $g$ are non-real, we shall call it  \emph{loxodromic}.  \end{definition}
 Let $A$ be a hyperbolic element in $\Sp(2,1)$. Let $\lambda$ represents an eigenvalue from the similarity class of eigenvalues $[\lambda]$ of $A$. Let $\x$ be a $\lambda$-eigenvector. Then $\x$ defines a point $x$  on $\H\P^2$ that is either a point on  $\partial \h^2$ or, a point in $\P(\V_{+})$. The lift of $x$ in $\H^{2,1}$ is the quaternionic line $\x \H$. We call $x$ as \emph{eigen-point} of $A$ corresponding to $[\lambda]$. Note that $x$ is a fixed point of $A$ in $\H \P^2 \setminus (\h^2 \cup \partial \h^2)$.

\medskip Let $A$ be a hyperbolic element then it has similarity classes of eigenvalues $[\lambda]$, $[\overline{\lambda}]^{-1}$ and $[\mu]$, where $|\lambda|<1$, $|\mu|=1$. Thus we can choose  eigenvalue representatives to be the complex numbers $re^{i\theta},~r^{-1}e^{i\theta}$, $e^{i\phi}$, $\theta, \phi \in [0, \pi]$, $0<r<1$. 

\medskip

 Let $a_A \in \partial\h^2$ be the \emph{ attracting fixed point } of $A$ that corresponds to the eigenvalue $re^{i \theta}$ and let $r_A \in \partial\h^2$ be the \emph{repelling fixed point} corresponding to the eigenvalue $r^{-1}e^{i \theta}$. Let $a_A$ and $r_A$ lift to eigenvectors $\a_A$ and $\r_A$ respectively.  Let $\x_A$ be an eigenvector corresponding to $e^{i \phi}$. The point $x_A$ on $\P(\V_+)$ is the \emph{polar-point} of $A$. For $(r,\theta, \phi)$ as above,  define $ E_A(r, \theta, \phi)$ as

\begin{equation}\label{li1}
  E_A(r, \theta, \phi)=\left[\begin{array}{ccc}
                         re^{i\theta} &  0& 0 \\
                         0 & e^{i\phi} & 0\\
                         0& 0& r^{-1}e^{i\theta}\\
                        \end{array}\right].\\
\end{equation}
 Let $C_A=\left[\begin{array}{ccc} \a_A & \x_A & \r_A\\
  \end{array}\right]$ be the $3\times 3$ matrix corresponding to the eigenvectors. We can choose $C_A$ to be an element of $\Sp(2,1)$ by normalizing the eigenvectors:
$$\langle \a_A, \r_A \rangle=1,~\langle \x_A, \x_A \rangle=1.$$
 Then $A=C_A E_A(r, \theta, \phi) C_A^{-1}$.  So, every hyperbolic element  $A$ in $\Sp(2,1)$ is conjugate to a matrix of the form  $E_A (r, \theta, \phi)$.

 For more details about conjugacy classification in $\Sp(2,1)$,  we refer to Chen-Greenberg \cite{cg}. We note the following fact that is useful for our purposes.

\begin{lemma}\label{hycl}{\rm (Chen-Greenberg )\cite{cg}}
 Two hyperbolic elements in ${\rm Sp}(2,1)$ are conjugate if and only if they have the same similarity classes of eigenvalues.
\end{lemma}
\begin{lemma}\label{emb}
The group $\Sp(2,1)$ can be embedded in the group ${\rm GL}(6, \C)$.
\end{lemma}
\begin{proof}
Write $\H=\C\oplus{\bf j}  \C$. For $A\in \Sp(2,1)$, express $A=A_1+{\bf j}A_2$,
{where} $ A_1, A_2\in M_{3}(\C)$. This gives an embedding $A \mapsto A_{\C}$ of $\Sp(2,1)$ into ${\rm GL}(6, \C)$, where
\begin{equation}\label{crep} A_{\C}= \left(
                          \begin{array}{cc}
                            A_1 &  -\overline{A_2}\\
                          {A_2}  & \overline{A_1} \\
                          \end{array}
                        \right).
\end{equation}
\end{proof}
 The following is a consequence of \lemref{hycl} and \lemref{emb}, for a proof see \cite{cago}.
\begin{prop}\label{rt}
Let $A$ be an element in $\Sp(2,1)$. Let $A_{\C}$ be the
corresponding element in ${\rm GL}(6, \C)$. The characteristic polynomials
of $A_{\C}$ is of the form
\begin{equation*}\chi_A(x)= x^6-a x^5+b x^4 -c x^3+bx^2-ax+1 =x^3 g(x+x^{-1}),\end{equation*}
where $a, b, c \in \R$. Let $\Delta$ be the negative of the discriminant of the polynomial $g_A(t)=g(x+x^{-1})$. Then the conjugacy class of  $A$ is determined by the real numbers $a, ~b, ~c$. 
\end{prop}

In the special case of $\Sp(1,1)$, counterparts of the above theorem may be obtained from the works \cite{gop} or  \cite{cpw}, also see \cite[Section 5.3]{gpp}.  Now we define the following notion. 
\begin{definition}
Let $A$ be a hyperbolic isometry of $\h^2$. The real tuple $(a, b, c)$ in  \propref{rt} will be called the  \emph{real trace} of $A$ and we shall denote it by $tr_{\R}(A)$.
\end{definition}

\subsection{The Cross Ratios}\label{crai}  
Given an ordered quadruple of pairwise distinct points $(z_1, z_2, z_3, z_4)$ on $\partial \h^2$, their Kor\'anyi-Reimann quaternionic cross ratio is defined by
$$\X(z_1, z_2, z_3, z_4)=[z_1,z_2,z_3,z_4]={\langle {\bf z}_3, {\bf z}_1 \rangle \langle {\bf z}_3, \bf z_2 \rangle}^{-1} { \langle {\bf z}_4, {\bf z}_2\rangle \langle   {\bf z}_4, {\bf z}_1 \rangle^{-1}},$$
where, for $i=1,2,3,4$,  ${\bf z}_i$, are lifts of $z_i$. Note that quaternionic cross ratios are not independent of the chosen lifts, but similarity classes of the cross ratios are independent of the chosen lifts. The conjugacy invariants obtained from the quaternionic cross ratios are $\Re(\X)$ and $|\X|$. Under the action of the symmetric group ${\rm S}_4$ on a tuple, there are exactly three orbits, see \cite[Prop 3.1]{platis}. This implies that the moduli and real parts of the quaternionic cross ratios are determined by the following three cross ratios.
$$\X_1(z_1, z_2, z_3, z_4)=\X(z_1, z_2, z_3, z_4)=[\z_1, \z_2, \z_3, \z_4],$$
$$\X_2(z_1, z_2, z_3, z_4)=\X(z_1, z_4, z_3, z_2)=[\z_1, \z_4, \z_3, \z_2],$$
 $$\X_3(z_1, z_2, z_3, z_4)=\X(z_2, z_4, z_3, z_1)=[\z_2, \z_4, \z_3, \z_1].$$
Furthermore, Platis proved that these cross ratios satisfy the following  real relations:
$$|\X_2|=|\X_1||\X_3|, ~ \hbox{and, }~~  2|\X_1|\Re(\X_3)=|\X_1|^2+|\X_2|^2-2\Re(\X_1)-2\Re(\X_2)+1.$$
For a given quadruple $(z_1, z_2, z_3, z_4)$ of $\partial \h^2$, the triple of cross ratios $(\X_1, \X_2, \X_3)$  corresponds to a point on the four dimensional real variety $\R^6$ subject to the above two real equations. This is called the  \emph{cross ratio variety}.

However, unlike the complex case, a point on the quaternionic cross ratio variety does not determine the $\Sp(2,1)$-congruence classes of ordered quadruples of points on $\partial \h^2$. This has been proven by  Cao in \cite{cao}. We shall recall Cao's result in  \thmref{cao} below.

\subsubsection{\bf Cartan's angular invariant}\label{cai} 
Let $z_1,~z_2,~z_3$ be three distinct points of $\overline{ \h^2}=\h^2 \cup \partial\h^2$, with lifts $\z_1,~\z_2$ and $\z_3$ respectively. The quaternionic Cartan's angular invariant associated to the triple $(z_1, z_2, z_3)$ was defined by Apanasov and Kim in \cite{ak}, and is given by the following:
$$\A(z_1,~z_2,~z_3)=\arccos\frac{\Re(-\langle \z_1, \z_2, \z_3\rangle)}{|\langle \z_1, \z_2, \z_3\rangle|},$$
where  $\langle \z_1, \z_2, \z_3\rangle=\langle \z_1, \z_2\rangle\langle \z_2, \z_3\rangle\langle \z_3, \z_1\rangle$. 
The angular invariant belongs to the interval  $[0, \frac{\pi}{2}]$. It is independent of the chosen lifts and also $\Sp(2,1)$-invariant. The following proposition shows that this invariant determines
any triples of distinct points on $\partial\h^2$ up to ${\rm Sp }(2,1)$-equivalence. For a proof see  \cite{ak}.
\begin{prop}\label{cai1}{\rm \cite{ak}}
Let $z_1,~z_2,~z_3$ and $z_1^\prime,~z_2^\prime,~z_3^\prime$ be triples of distinct points of $\overline{ \h^2}$. Then $\A(z_1,~z_2,~z_3)=\A(z_1^\prime,~
z_2^\prime,~z_3^\prime)$ if and only if there exist $A\in {\rm Sp }(2,1)$ so that $A(z_j)=z_j^\prime$ for $j=1,2,3$.
\end{prop}
Further, it is proved in \cite{ak} that $(z_1, z_2, z_3)$ lies on the boundary of an $\H$-line, resp. a totally real subspace, if and only if $\A= \frac{\pi}{2}$, resp.  $\A=0$.

\subsection{$\Sp(2,1)$-congruence classes of ordered quadruples of points on $\partial \h^2$}  We shall use the following result by Cao that determines ordered quadruples of points on  $\partial \h^2$, see  \cite{cao}. 
\begin{theorem}\label{cao} {\rm  \cite{cao}}
Let $Z=(z_1, z_2, z_3, z_4)$ and $W=(w_1, w_2, w_3, w_4)$ be two ordered quadruples of pairwise distinct points in $\partial \h^2$. Then there exists an isometry $h \in \Sp(2,1)$ such that $h(z_i)=w_i$, $i=1,2,3,4$, if and only if the following conditions hold:
\begin{enumerate}
\item For $j=1,2,3$, $\X_j(z_1, z_2, z_3, z_4)$ and $~\X_j(w_1, w_2, w_3, w_4)$ belong to the same similarity class.

\item $\A(z_1, z_2, z_3)=\A(w_1, w_2, w_3)$, ~$\A(z_1, z_2, z_4)=\A(w_1, w_2, w_4)$, ~ $\A(z_2, z_3, z_4)=\A(w_2, w_3, w_4)$.
\end{enumerate}
\end{theorem}

Cao has described the moduli of $\Sp(2,1)$-congruence classes of ordered quadruples of distinct points on $\partial \h^2$. We recall his result here. 
\begin{theorem} {\rm \cite{cao}} \label{cao2}
Let ${\M}(2)$ be the configuration space of ordered quadruples of distinct points on $\partial \h^2$. Then ${\M}(2)$ is homeomorphic to the semi-analytic subspace of $\C^3 \times [0, \infty) \times [0, \frac{\pi}{2}]$ given by points $(c_1, c_2, c_3, t, \A)$ subject to the relations
$$\Re(c_1 \bar c_2) + t \Re(c_3) \leq 0, ~ \Re(c_2) \leq 0, |c_1|^2 +t^2 \neq 0, ~ |c_2|^2 +|c_3|^2 \neq 0, $$
$$1+|c_1|^2+|c_2|^2+|c_3|^2+t^2-2 \Re(c_1)+2 \Re(c_2 e^{-i \A})+2 \Re((\bar c_1 c_2+t c_3)e^{i \A}) =0.$$
\end{theorem}

\medskip Note that there is a natural action of the symmetric group $S_4$ on $\M(2)$, coming from the $S_4$ action on an ordered tuple:  
$$g. [(p_1, p_2, p_3, p_4)]=[(p_{g(1)}, p_{g(2)}, p_{g(3)}, p_{g(4)})].$$
The orbit space of $\M(2)$ under this action will be denoted by $\mathcal M(2)$. 

\section{ Projective Points}\label{proj}
In this section, we introduce the notion of \emph{projective points} of an eigenvalue. This is a crucial notion for our understanding of the pairs of hyperbolic elements. 

\medskip Let $T$ be an invertible matrix over $\H$. Let $\lambda \in \H\setminus\R$ be a chosen eigenvalue of $T$ in the similarity class $[\lambda]$.  We assume that $[\lambda]$ has geometric multiplicity one, i.e. the $[\lambda]$-eigenspace has dimension one. Thus, we can identify the eigenspace with $\H$.  Consider the $\lambda$-\emph{eigenset}: $S_{\lambda}=\{ x \in \V \ | \ Tx =x \lambda \}$. Note that this set is the complex line $x Z(\lambda)$ in $\H$.

\medskip Identify $\H$ with $\C^2$. Two non-zero quaternions $q_1$ and $q_2$ are equivalent if  $q_2=q_1 c$, $c \in \C \setminus 0$. This equivalence relation projects $\H$ to the one dimensional complex projective space $\C \P^1$. 

\medskip Let $v$ be a $\lambda$-eigenvector of $T$. Then ${\ve} \H$ is the quaternionic line spanned by ${\ve}$.  Now, we identify ${\ve} \H$ with $\H$. Then for each point on $\C \P^1$, there is a choice of the lift ${\ve}$ of $v$ that spans a complex line in ${\ve}\H$. The point on $\C \P^1$ that corresponds to a specified choice of ${\ve}$ is called a \emph{projective point} of $T$ corresponding to the eigenvalue $\lambda  \in \H\setminus\R$. A $[\lambda]$-projective point of $T$ corresponds to an eigenset of an eigenvalue representative $\lambda$, equivalently, to the centralizer $Z(\lambda)$.   The  $\C\P^1$ obtained as above from the eigenspace ${\ve} \H$ is called a \emph{$[\lambda]$-eigensphere }of $T$.  Since $[\lambda]$ is a conjugacy invariant of $T$, so also the $[\lambda]$-eigensphere $\C \P^1$.

\medskip To see the projective points from another viewpoint, we note the following. Consider the action of $\Sp(1)$ on $\H^{\ast}$ by conjugation. The similarity class of an eigenvalue  $[\lambda]$ represents an orbit under this action. The stabilizer of a point under this action is $\SU(1)$. Hence $[\lambda]$ is identified with $\C \P^1$ that is the orbit space $\Sp(1)/\SU(1)$. The choice of $\lambda$ on this $\C \P^1$ is a  $[\lambda]$-\emph{projective point} of $T$. 

\medskip If $\lambda \in \R \setminus \{0\}$, then it commutes with every quaternion, and hence $Z(\lambda)=\H$. Consequently, there is only one eigenset of $\lambda$ that equals the eigenspace. We may assume the $\lambda$-eigensphere to be  a single point in this case. 

\section{Classification of Hyperbolic elements in $\Sp(2,1)$.} \label{hyp}

For clarity, first we define the following notion. 
\begin{definition}
Let $A$ and $A'$ be two hyperbolic elements having a common eigenvalue $[\alpha]$, $\alpha \in \H \setminus \R$. Then $A$ and $A'$ are said to have the same $[\alpha]$-projective point if they have the same projective point with respect to a representative $\alpha$ of $[\alpha]$.  
\end{definition} 
Suppose $A$ and $A'$ are two hyperbolic elements having the same $[\alpha]$-projective point. If $A$ and $A'$ have the same eigenset with respect to a chosen representative $\alpha$, then they have the same eigensets with respect to all other representatives. In this case $A$ and $A'$ define the same projective point for all representatives of  $[\alpha]$.

\medskip Now, we have the following lemma that determines hyperbolic elements in $\Sp(2,1)$. 
\begin{lemma}\label{lox}
Let $A$ and $A'$ be hyperbolic elements in $\Sp(2,1)$. Then $A=A'$ if and only if they have the same attracting fixed point, the same repelling fixed point, the same real trace,  and the same projective points.
\end{lemma}
\begin{proof}
If $A=A'$, then the statement is clear. We prove the converse. Without loss of generality, we can assume $A$ and $A'$ to be loxodromic elements. The other cases are similar.

Let $A$ and $A'$ be two loxodromic elements of $\Sp(2,1)$. Since they have the same real traces,  let the eigenvalue representatives of $A$ and $A'$ be $re^{i \theta}$, $r^{-1} e^{i \theta}$ and $e^{i \phi}$, where $0<r<1$, $\theta, \phi \in (0, \pi)$. Then
$$A=C_AE_A(r, \theta, \phi) C_A^{-1}, ~~\hbox{ and } A'=C_{A'} E_A(r, \theta, \phi) C_{A'}^{-1}.$$
Since $A$ and $A'$ have the same fixed-points, hence $\a_{A'}=\a_A q_1$, $\r_{A'}=\r_A q_2$, and $\x_{A'}=\x_A q_3$, where $q_j \in \H \setminus 0$.  Let
$$D=\begin{pmatrix} q_1 & 0 & 0 \\ 0 & q_3 & 0 \\0 & 0 & q_2 \end{pmatrix}.$$
Then $A=C_A E_A(r, \theta, \phi) C_A^{-1}$ and, $A'=C_A D E_A(r, \theta, \phi) D^{-1} C_A^{-1}$. Now $D$ commutes with $E_A(r, \theta, \phi)$ if and only if $q_1,q_2,q_3$ are complex numbers. This is equivalent to the condition of having the same projective points. Thus, it follows that $A=A'$. 
\end{proof}

\subsection{Projective parameters of a loxodromic}
 Suppose $A$ is a loxodromic element in $\Sp(2,1)$.
If $a_A$ and $r_A$ are the fixed-points of $A$, then they are joined by a complex line
$$\a_A Z(\lambda)  +\r_A Z(\lambda),$$
and hence $(\a_A, \r_A)$ is determined by a single projective point on $\C \P^1$ that corresponds to $Z(\lambda)$. Here we have used the fact that $Z(\lambda)=Z(\bar \lambda^{-1})$. Similarly, the projective point of $\x_A$ corresponds to the centralizer of $Z(\mu)$.

\medskip  Thus given a triple $(a_A, r_A, x_A) \in \partial \h^2 \times \partial \h^2 \times \P(\V_+)$, and real number $r$, there is a  two complex parameter family $\mathcal H_r$ of loxodromic elements having the same real trace $r$ and fixing points $(a_A, r_A, x_A)$.  This two complex family of parameters correspond to the projective points of $A$ on  $\C \P^1 \times \C \P^1$. Thus $\mathcal H_r$ is parametrized by $\C \P^1 \times \C \P^1$.  
Given $A \in \mathcal H_r$, the space $\C \P^1 \times \C \P^1$ corresponds to the conjugacy class of $A$ in the stabilizer subgroup $\Sp(2,1)_{(a_A, r_A)}.$ 

\subsection{Parametrization of conjugacy classes of loxodromics in $\Sp(2,1)$} \label{type} When $A$ is a loxodromic element in $\Sp(2,1)$, it follows that $tr_{\R}(A)$ belongs to an open subspace  ${\rm D}_2$  of $\R^3$. Using \propref{rt} we deduce the following lemma that  provides a description of ${\rm D}_2$.
\begin{lemma}\label{dol}
Let $A$ be  hyperbolic in $\Sp(2,1)$ and let $A_{\C}$ be its complex representative. The characteristic polynomials
of $A_{\C}$ is of the form
\begin{equation*}\chi_A(x)=x^6-ax^5 +bx^4-cx^3+bx^2-ax+1,\end{equation*}
where $a$, $b$, $c$ are real numbers.
Define $$G=27(a-c)+9ab-2a^3,$$
$$H=3(b-3)-a^2,$$
$$\Delta=G^2 +4H^3.$$
Then $A$ is loxodromic if and only if $\Delta >0$, $|2a+c| \neq |2b+2|$.
\end{lemma}
\begin{proof}
Since $A$ is hyperbolic, it has eigenvalue representatives of the form $re^{i \theta}$, $r^{-1} e^{ i \theta}$ and $e^{i \phi}$, where $r>0,r\neq 1$ and $\theta ,\phi \in [0,\pi].$ Further note that  if $\alpha$ is a root of $\chi_A(x)$, then $\alpha+\alpha^{-1}$ is a root of $g_A(t)$. Thus we obtain
\begin{equation}\label{1}
g_A(t)= t^3-at^2+(b-3)t-(c-2a),
\end{equation}
where
\begin{eqnarray*}\label{loa1}
  a&=&2 \big(r+\frac{1}{r}\big)\cos\theta+2\cos\phi,\\
b&=&4\big(r+\frac{1}{r}\big)\cos\theta \cos\phi+4\cos^2\phi
+r^2+\frac{1}{r^2}+1,\\
c&=&4\big(r+\frac{1}{r}\big)\cos\theta+
2\big(r^2+\frac{1}{r^2}+4\cos^2\theta \big)\cos\phi.
\end{eqnarray*}
 Now to detect the nature of roots of the cubic equation, we look at the discriminant sequence $(G, H, \Delta)$ of $g_A(t)$. The multiplicity of a root of
$g_A(t)$ is determined by the resultant $R(g, g'')$ of $g_A(t)$ and its second derivative $g''_A(t)=6t-2a$. Note    that
$$R(g, g'')=-8[27(a-c)+9ab-2a^3]=-8G.$$
 When $A$ is loxodromic, $g_A(t)$ has the following roots:
$$t_1=re^{{ i} \theta}+r^{-1}e^{{- i} \theta}, t_2=r^{-1}e^{ i \theta}+re^{- i 
\theta}, t_3=2\cos\phi.$$
Thus in this case $\Delta>0$ and $g_A(\pm 2) \neq 0$.  In all other cases, either, $\Delta>0$ and $g_A(\pm 2)=0$,  or,  $\Delta=0$.  This proves the lemma. 
\end{proof}

For $ G=27(c-a)-9ab+2a^3$,  $H=3(b-3)-a^2$,
 $${\rm D}_2=\{(a, b, c ) \in \R^3 \ | \ G^2 + 4H^3>0, ~|2a+c| \neq |2b+2| \}.$$
For each loxodromic $A$, $tr_{\R}(A)$ is an element of ${\rm D}_2$. Conversely, given an element $(a, b, c)$  from the set ${\rm D}_2$, we have a conjugacy class of loxodromics $A$ with $tr_{\R}(A)=(a, b, c)$. Thus we have the following consequence of the above lemma.
\begin{corollary}
The set of conjugacy classes of loxodromic elements in $\Sp(2,1)$ can be identified with ${\rm D}_2$. 
\end{corollary}
 
\begin{remark}\label{rtype}
Other than loxodromics, there are three more types of hyperbolic elements in $\Sp(2,1)$.  These  types correspond to the cases:

(i) $\Delta >0$, $|2a+c| =|2b+2|$, and

(ii) $\Delta =0$, $|2a+c| \neq |2b+2|$. 

(iii) $\Delta =0$, $|2a+c| = |2b+2|$.

\medskip In the case (i), the hyperbolic element has only one real eigenvalue,  in the case (ii), there are only two real eigenvalues, and in case (iii), all the eigenvalues are real numbers, i.e. the element is strictly hyperbolic.   For hyperbolic elements of types (i) and (ii),  the real traces are parametrized by a two real parameter family. 
\end{remark}
\section{Proof of \thmref{plox}}\label{palo}

\begin{proof}
Suppose $(A, B)$ and $(A', B')$ be pairs of loxodromics such that $tr(A)=tr(A')$, $tr(B)=tr(B')$, for $i=1,2,3$,  $[\X_i(A, B)]=[\X_i(A', B')]$, $\A_i(A, B)=\A_i(A', B')$. Since the cross ratios and angular invariants are equal, by \thmref{cao}, there is an element $C$ in $\Sp(2,1)$ such that $C(a_A)=a_{A'}$, $C(r_A)=r_{A'}$, $C(a_B)=a_{B'}$ and $C(r_B)=r_{B'}$. Therefore $A'$ and $CAC^{-1}$ have the same fixed points. Since they also have the same real traces, and the same projective points, we have  by \lemref{lox} that $CAC^{-1}=A'$. Similarly, $CBC^{-1}=B'$. This completes the proof.
\end{proof}

\subsection{Proof of \corref{cormix}} In this case $A$ is loxodromic and $B$ is strictly hyperbolic. The same proof as above works noting that for $B$ we do not require any projective points, as the real eigenvalues are elements of centralizers in $\H$. 

\medskip Similarly, \corref{corsh} follows, and also similar results can be deduced for other types of mixed hyperbolic pairs.

\subsection{Proof of \corref{cord1}}
\begin{proof}
Let $\F_2=\langle x, y \rangle$.  Given a representation $\rho$,  we have the following correspondence from $\dchi_o(\F_2, \Sp(2,1))$ onto ${\rm D}_2 \times {\rm D}_2 \times {\mathcal M}(2)$:
$$\mathfrak f:  [\rho] \mapsto (tr_{\R}(\rho(x)), tr_{\R}( \rho(y)), [(a_{\rho(x)}, r_{\rho(x)}, a_{\rho(y)}, r_{\rho(y)})] ).$$
Given a point on ${\rm D}_2 \times {\rm D}_2 \times \mathcal M(2)$, it follows from \lemref{lox} that the inverse-image of the point under $\mathfrak f$  is $(\C \P^1 \times \C \P^1) \times (\C \P^1 \times \C \P^1)$ corresponding to the projective points of $(\rho(x), \rho(y))$.   This completes the proof.
\end{proof}

\bigskip 
\begin{remark} 
{\it Change of coordinates on the same three-holed sphere and `trace' coordinates}. From a group theoretic viewpoint, a three-holed sphere corresponds to a subgroup generated by two hyperbolic elements $A$ and $B$ whose product $AB$ is also loxodromic. The three boundary curves corresponds to the loxodromic elements $A$, $B$ and $B^{-1} A^{-1}$ respectively. A group generated by such hyperbolic elements is called a $(0, 3)$ subgroup. The classical Fenchel-Nielsen coordinates are obtained by `gluing' such $(0,3)$ subgroups, and the coordinates are given by several parameters associated to these subgroups and the gluing process. 

\medskip There is a natural three-fold symmetry associated to a $(0,3)$ subgroup. This is respected in the classical Fenchel-Nielsen coordinates of the Teichm\"uller or the quasi-Fuchsian space.  For example, if we change the coordinates associated to $\langle A, B \rangle$ to $\langle A, B^{-1} A^{-1} \rangle$, then the Fenchel-Nielsen coordinates remain unchanged in the classical set up. This is clearly not the case in our set up, neither it was in the set up of Parker and Platis in the complex hyperbolic set up, see \cite[Section 7.2]{pp}. However, Parker and Platis rectified this problem partially by relating the traces of $A$, $B$ and several of their compositions with the cross ratios, and thus by giving a real analytic change of coordinates between two $(0, 3)$ groups coming from the same three-holed sphere. Following the work of Will \cite{will2}, Parker resolved this problem by using trace parameters to determine an irreducible $(0, 3)$ subgroup of $\SU(2,1)$ in \cite{par}. 

\medskip We do not know how to resolve this issue in the $\Sp(2,1)$ set up.  We expect that by computations it should be possible to relate the real traces and the points on $\M(2)$ of two $(0,3)$ subgroups coming from the same three-holed sphere. The computations to do that are too involved, and we are unable to resolve the difficulty here. It is also not clear to us that how the projective points change when we have a change of the $(0,3)$ subgroups. 

As mentioned in the Introduction,  innovation of a set of  `real trace  coordinates' using classical invariant theory might also be helpful  to resolve the above problem concerning the three-fold symmetry of a three-holed sphere.
 \end{remark} 

\section{Proof of \thmref{mnth2}}\label{fn}
To prove the theorem, we need to determine the parameters that are needed while attaching two $(0, 3)$ groups, or `closing a handle'. To obtain Fenchel-Nielsen coordinates on the representation variety, we may need to attach two $(0,3)$ groups to yield an $(1,1)$ group. Here an $(1, 1)$ group is a subgroup of $\Sp(2,1)$ generated by two loxodromic elements and their commutator. For this reason, we need to define the \emph{twist-bend parameters} while attaching $(0,3)$ groups. We follow similar ideas as noted in the paper of Parker and Platis \cite{pp}, and will also use the same terminologies given there. For detailed information about the terminologies and ideas inbuilt in the process, we refer to \cite{pp} and \cite{sw}.  We shall only sketch those parts from the scheme of attaching two $(0,3)$ groups that are not apparent in the $\Sp(2,1)$ setting, and deserves mention for clarity of the exposition. In the following all the $(0,3)$ groups will be assumed to be irreducible. 

\subsection{Twist-bend parameters}
\medskip Let $\langle A, B \rangle$ and $\langle C, D \rangle$ be two $(0, 3)$ groups such that their \emph{boundaries are compatible}, that is, $A=D^{-1}$. A \emph{quaternionic hyperbolic twist-bend} corresponds to an element $K$ in $\Sp(2,1)$ that commutes with $A$ and conjugate $\langle C, D \rangle$.  Up to conjugacy assume $A$ fixes $0$, $\infty$ and it is of the form $E_A(r, \theta, \phi)$. Since $K$ commutes with $A$, it is also of the form $E_K(s, \psi, \xi)$, $s \geq 1$, $\psi, \xi \in[0, \pi]$. If $s \neq 1$, from \lemref{lox}, it follows that there is a total of seven real parameters associated to $K$, the real trace $(s, \psi, \xi)$, along with four real parameters associated to the projective points. If $s=1$, then $K$ is a boundary elliptic and the eigenvalue $[e^{i \psi}]$ has multiplicity $2$. But, we can still define the projective points for these eigenvalues.  There are exactly one negative-type and one positive-type eigenvalues of $K$ in this case. Since $K$ commutes with $A$, the projective points of $K$ is determined by the projective points of $A$. Hence,  there are two projective points of $K$ to determine it. Consequently, we shall have seven parameters associated to a twist-bend $K$. We denote these parameters by $\kappa=(s, \psi, \xi, k_1, k_2)$, where $k_1=p_1(K)$, $k_2=p_2(K)$ are the projective points of the  similarity classes of eigenvalues of $K$. 

We further fix the convention of choosing the twist-bend parameters such that it is \emph{oriented consistently with $A$}, i.e. if $A=QE_A(r, \theta, \phi) Q^{-1}$, then $K=Q E_K(s, \psi, \xi)Q^{-1}$. Since $\langle A, B\rangle$ and $\langle C, A^{-1}\rangle$ are considered irreducible, without loss of generality we assume that $a_B$, $r_C$ do not lie on a totally geodesic subspace joining $a_A$ and $r_A$. To obtain conjugacy invariants to measure the twist-bend parameter we define the following quantities.

$$\tilde \X_1(\kappa)=\X_1( a_A, r_A, K(r_C), a_B),$$
$$\tilde \X_2(\kappa)=\X_2(a_A, r_A, K(r_C), a_B),$$
$$\tilde \A_1(\kappa)=\A(a_A, r_A, K(r_C)), ~~\tilde \A_3(\kappa)=\A(r_A, K(r_C), a_B).$$

\medskip
\begin{lemma}\label{tlem}
Let $A$, $B$, $C$ be loxodromic transformations of $\h^2$ such that $\langle A, B \rangle$ and $\langle A^{-1}, C \rangle$ are irreducible  $(0, 3)$ subgroups of $\Sp(2,1)$.  Let $K=E_K(s, \psi, \xi, k_1, k_2)$ and $K'=E_{K'}(s', \psi', \xi', k_1', k_2')$ represent twist-bend parameters that are oriented consistently with $A$. If

$$[\tilde \X_1(\kappa)]=[\tilde \X_1(\kappa')], ~~[\tilde \X_2(\kappa)]=[\tilde \X_2(\kappa')] , \hbox{ and }$$

$$\tilde \A_1(\kappa)=\tilde \A_1(\kappa'), ~\tilde \A_3(\kappa)=\tilde \A_3(\kappa'),$$
and $k_1=k_1'$, $k_2=k_2'$,
then $K=K'$.
\end{lemma}

\begin{proof}
Without loss of generality we assume $a_A=o$, $r_A=\infty$.
In view of the conditions
$$[\tilde \X_1(\kappa)]=[\tilde \X_1(\kappa')], ~~[\tilde \X_2(\kappa)]=[\tilde \X_2(\kappa')] , \hbox{ and }$$
$$\tilde \A_1(\kappa)=\tilde \A_1(\kappa'), ~\tilde \A_3(\kappa)=\tilde \A_3(\kappa').$$
and by using real relations of cross ratio, we get $[\tilde \X_3(\kappa)]=[\tilde \X_3(\kappa')]$. The quantities $\tilde \A_2(\kappa)$ and $\tilde \A_2(\kappa')$ are trivially equal.  Now, following similar arguments as in the proof of \cite[Theorem 5.2]{cao}, we have $f$  in $\Sp(2,1)$ such that $f(a_A)=a_A$, $f(r_A)=r_A$, $f(a_B)=a_B$ and $f(E_K(r_C))=E_{K'}(r_C)$. Further, since it fixes three points on the boundary, it must be of the form  
$$f=\begin{bmatrix} \mu & 0 & 0 \\ 0 & \mu & 0 \\0 & 0 & \psi \end{bmatrix}.$$ Since $a_B$ does not lie on a totally geodesic subspace joining $a_A$ and $r_A$, we must have $\mu=\pm 1=\psi$. Thus, it follows that $E_{K}(r_C)= E_{K'}(r_C)$. Now by using the fact that $E_K {E_{K'}}^{-1}$ has the three fixed points $a_A= o, r_A=\infty $ and $r_C$ together with the condition that $r_C$ does not lie on a totally geodesic subspace joining $a_A$ and $r_A$, we have $E_K(s, \psi, \xi) =E_{K'}(s', \psi', \xi')$ i.e,\ $s=s',\psi=\psi',\xi=\xi'.$

 Hence, $K$ and $K'$ are conjugate with the same attracting and the same repelling points. So, by \lemref{lox}, $K=K'$ if and only if they have the same projective points and  the same fixed points. This completes the proof.
\end{proof}

In view of \thmref{plox} and \lemref{tlem}, we shall now proceed to prove \thmref{mnth2}.  The strategy of the proof is similar to the proofs given by Parker and Platis in \cite[Section 8]{pp}, or Gongopadhyay and Parsad in \cite[Section 6]{gp2}. The main challenge in the quaternionic set up was the derivations of \thmref{plox} and \lemref{tlem}. After those are in place, the rest follows mimicking arguments of Parker and Platis with appropriate modifications in the quaternionic set up.  These arguments are mostly group theoretic and does not involve arithmetic of the underlying quaternionic  algebra. We will only mention the key points of the attaching process. For detailed ideas behind them, we refer to the original article of Parker and Platis \cite{pp}. 
\subsubsection{Attaching two three-holed spheres} A $(0, 4)$ subgroup of ${\Sp}(2,1)$ is a group with four loxodromic generators such that their product is identity. These four loxodromic maps correspond to the boundary curves of the four-holed spheres and are called \emph{peripheral}.  Thus a $(0,4)$ group is freely generated by any of these three loxodromic elements.
Let $\langle A, B \rangle$ and $\langle C, D \rangle$ be two $(0,3)$ groups with $A^{-1}=D$. Algebraically, a $(0,4)$ group is constructed by the amalgamated free product of these groups with amalgamation along the common cyclic subgroup $\langle A \rangle$. Conjugating $\langle C, D \rangle$ by the twist-bend $K$ yields a new $(0, 4)$ subgroup that is depended on $K$.  Varying this $K$ gives the twist-bend deformation.

The following lemma goes the same way as Lemma 8.3 in \cite[p.131]{pp}.

\begin{lemma}
Suppose $\Gamma_1=\langle A, B \rangle$ and $\Gamma_2=\langle C, D \rangle$ are two irreducible $(0,3)$ groups with peripheral elements $A$, $B$, $B^{-1} A^{-1}$ and $C$, $D$, $D^{-1} C^{-1}$ respectively. Moreover suppose that $A=D^{-1}$. Let $K$ be any element of $\Sp(2,1)$ that commutes with $A=D^{-1}$. The the group $\langle A, B, KCK^{-1} \rangle$ is a $(0, 4)$ group with peripheral elements $B$, $B^{-1} A^{-1}$, $KCK^{-1}$ and $KD^{-1} C^{-1} K^{-1}$.
\end{lemma}

\begin{prop}\label{tb1}
Suppose that $\langle A, B \rangle$ and $\langle C, A^{-1} \rangle$ are two irreducible $(0,3)$ subgroups of $\Sp(2,1)$. Let $\kappa=(s, \psi, \xi, k_1, k_2)$ be a twist-bend parameter oriented consistently with $A$ and let $\langle A, B, KCK^{-1}\rangle$ be the corresponding $(0, 4)$ group. Then $\langle A, B, KCK^{-1} \rangle$ is uniquely determined up to conjugation in $\Sp(2,1)$ by the \emph{Fenchel-Nielsen coordinates}:
 $$\hbox{the real traces: }tr_{\R}(A),~tr_{\R}(B),~tr_{\R}(C);$$
$$\hbox{the  cross ratios: }[\X_{k}(A,~B)], ~[\X_k(A, C)],  k=1,2;$$
$$\hbox{the angular invariants: } \A_j(A, B), ~ \A_j(A, C),~j=1,2,3;$$
$$\hbox{six projective points: }p_i(A), ~p_i(B), ~ p_i(C),~i=1,2;$$
$$ \hbox{and the twist bend } \kappa=(s, \psi, \xi, k_1, k_2).$$
Thus we need a total of 42 real parameters to specify $\langle A, B, KCK^{-1}\rangle$ uniquely up to conjugacy.
\end{prop}

\medskip In the parameter space associated to $\langle A, B , KCK^{-1}\rangle$, the parameters corresponding to the traces have real dimension $6$, the cross ratios and angular invariants add up to $16$ degrees of freedom, six projective points add up to $12$ degrees of freedom and the twist parameter has $7$ degrees of freedom adding to a total of: $(3 \times 3)+(2 \times 4)+(1 \times 6) +(2 \times 6)+7=42$ degrees of freedom.

\begin{proof}
Suppose, $\langle A, B, KCK^{-1} \rangle$ and $\langle A', B', K'C'K'^{-1}\rangle$ are two $(0,4)$ subgroups having the same Fenchel-Nielsen coordinates. Let the given relations hold.
Using these relations, it follows from \thmref{plox} that there exist $C_1$ and $C_2$ in $\Sp(2,1)$ that conjugate $\langle A, B \rangle$ and $\langle C, A^{-1} \rangle$ respectively to $\langle A', B'\rangle$ and $\langle C', A'^{-1} \rangle$.

  Now the twist-bends are defined with respect to the same initial group $\langle A, B, C \rangle$ that we fix at the beginning before  attaching the two $(0,3)$ groups. So,  without loss of generality, we may assume that $A=A', ~ B=B', ~ C=C'$, and thus,  $C_1=C_2$. Now with respect to the same initial  group $\langle A, B, C \rangle$, by \lemref{tlem} it follows that $K=K'$. This implies that $\langle A, B, KCK^{-1} \rangle$ is determined uniquely up to conjugacy.

Conversely, suppose that $\langle A, B, KCK^{-1} \rangle$ and $\langle A', B', K'C'K'^{-1} \rangle$ are conjugate, and having the same projective points.  Then clearly, the traces and angular invariants are equal and, the cross ratios are similar. The only thing remains to show is the equality of the twist-bends. Now by the invariance of the cross ratios it follows that
$$[\tilde \X_1(\kappa)]=[\tilde \X_1(\kappa')], ~[\tilde \X_2(\kappa)]=[\tilde \X_2(\kappa')], ~ \tilde \A_1(\kappa)=\tilde \A_1(\kappa'), ~\tilde \A_3(\kappa)=\tilde \A_3(\kappa').$$
and hence by \lemref{tlem}, $\kappa=\kappa'$.
\end{proof}
\subsection{Closing a handle} The next step is to  obtain a one-holed torus by attaching two holes of the same three-holed sphere in the quaternionic hyperbolic plane. The process of attaching these two holes is called \emph{closing a handle}. Geometrically, it corresponds to attach two boundary components of the same three-holed sphere. From a group theoretic viewpoint, closing a handle is the same as taking the HNN-extension of the $(0, 3)$ group $\langle A, BA^{-1} B^{-1} \rangle$ by adjoining the element $B$ to form a $(1,1)$ group. When we take the HNN-extension, the map $B$ is not
unique. If $K$ is any element in $\Sp(2,1)$ that commutes with $A$, then $\langle A, BK \rangle$ gives another $(1,1)$ group. Varying $K$ corresponds to a twist-bend coordinate as above.

 If $A=QE(s, \psi, \xi) Q^{-1}$ for $(s, \psi, \xi)$, just as before, we define the twist-bend parameter  $\kappa$ by
$K=QE(s, \psi, \xi, k_1,k_2) Q^{-1}$, and we say, $\kappa=(s, \psi, \xi, k_1,k_2)$ is oriented consistently with $A$. In this case also $\kappa$ is defined relative to a reference group that we fix at the beginning of the attaching procedure.

\begin{lemma}\label{ch}
Let $\langle A, BA^{-1} B^{-1} \rangle$ be a irreducible $(0, 3)$ group. Let $B$ be a fixed choice of an element in $\Sp(2,1)$ conjugating $A^{-1}$ to $BA^{-1} B^{-1}$. Let  $\kappa=(s, \psi, \xi, k_1,k_2)$ and  $\kappa'=(s', \psi', \xi', k_1,k_2')$  be twist-bend parameters oriented consistently with $A$. Then $\langle A, BK\rangle$ is  conjugate to $\langle A, BK'\rangle$ if and only if $\kappa=\kappa'$.
\end{lemma}
\begin{proof} 
If $\kappa=\kappa'$, then clearly $K=K'$ and hence the groups are equal. 

Conversely, suppose $\langle A, BK \rangle$ is conjugate to $\langle A, BK'\rangle$. The element $D$ that conjugates these groups, must commutes with $A$. Hence, 
$$D(a_A)=a_A, ~ D(r_A)=r_A.$$  Since $BA^{-1} B^{-1}$ has been fixed at the beginning, we have
$$BA^{-1} B^{-1}=(BK') A^{-1} (BK')^{-1}=D(BA^{-1} B^{-1})D^{-1}.$$
Thus $D$ commutes with $BA^{-1} B^{-1}$ and fixes $a_{BA^{-1} B^{-1}}=B(r_A), ~ r_{BA^{-1} B^{-1}}=B(a_A)$. Since the fixed points are distinct, $D$ is either the identity or,  the fixed points $a_A, ~ r_A, ~B(a_A),~B(r_A)$ belong to the boundary of the same totally geodesic subspace fixed by $D$.  But the later is not possible by the irreducibility of the $(0,3)$ group. So,  $D$ must be the identity,  $BK'=BK$ and hence, $K=K'$, i.e., $\kappa=\kappa'$. 
\end{proof} 

\begin{prop}\label{tw2}
Let  $\langle A, BK \rangle$ be a $(1,1)$ group obtained from the irreducible $(0,3)$ group $\langle A, BA^{-1} B^{-1} \rangle$ by closing a handle with associated twist-bend parameter $\kappa$. Then $\langle A, BK \rangle$ is determined up to conjugation by its \emph{Fenchel-Nielsen coordinates}
 $$tr_{\R}(A), ~[\X_{j}(A, BA^{-1} B^{-1})],~\A_{j}(A, BA^{-1} B^{-1}),~j=1,2,3,~ p_1(A),~p_2(A),$$
 and the twist-bend parameter $\kappa=(s, \psi, \xi, k_1, k_2)$.

\medskip Thus,  we need 21 real parameters to specify $\langle A, BK \rangle$ up to conjugacy.
\end{prop}
\begin{proof}
Suppose that $\langle A, BK \rangle$ and $\langle A, B'K' \rangle$ are two $(1,1)$ groups with the same Fenchel-Nielsen coordinates. In particular $tr(A)=tr(A')$ and hence
$$tr(BA^{-1} B^{-1})=\overline{ tr(A)}=\overline{tr(A')}=tr(B'A'^{-1} B^{-1}).$$
Further using the following relations
$$[\X_k(A, BA^{-1} B^{-1})]=[\X_k(A', B'A'^{-1} B'^{-1})],$$
 $$[\X_k(A, BA^{-1} B^{-1})]=[\X_k(A', B'A'^{-1} B'^{-1})], ~k=1,2,3,$$
we see by \thmref{plox} that the $(0, 3)$ groups $\langle A, BA^{-1} B^{-1} \rangle$ and $\langle A', B'A'^{-1} B'^{-1} \rangle$  are determined by the projective points of $A$. Thus we can assume $A=A', ~  BA^{-1} B^{-1}=B'A'^{-1} B'^{-1}$. Now using the above lemma,  we see that $\kappa=\kappa'$. Hence $K=K'$. Thus, the group $\langle A, BK\rangle$ is determined uniquely up to conjugation.

Conversely, suppose $\langle A, BK \rangle$ and $\langle A', B'K' \rangle$ are conjugate, then it is clear that the unique conjugacy class is determined by the given invariants.
\end{proof}
\subsection{Proof of \thmref{mnth2}} \label{smnth2}
\begin{proof}
Let $\Sigma_g \setminus \mathcal C$ be the complement of the curve system $\mathcal C$ in $\Sigma_g$. This is a disjoint union of $2g-2$ three holed spheres. Each such three-holed sphere in  $\Sigma_g \setminus \mathcal C$ corresponds to an irreducible $(0, 3)$ subgroup of ${\Sp}(2,1)$. By \thmref{plox}, a $(0, 3)$ subgroup $\langle A, B \rangle$ is determined up to conjugacy by the $21$ real parameters given there.  While attaching two three-holed spheres, we attach two $(0, 3)$ groups subject to the compatibility condition that a peripheral element in one group is conjugate to the inverse of a peripheral element in the other group. This gives a $(0, 4)$ group that is specified up to conjugacy by the $42$ real parameters described in  \propref{tb1}.
Proceeding this way, attaching  $2g-2$ of the above $(0, 3)$ groups, we get a surface with $2g$ handles, and it is determined by  $21(2g-2)=42g-42$ real parameters obtained from the attaching process. The handles correspond to the $g$ curves that in turn correspond to the two boundary components of the  three-holed spheres. 
 
 Now,  there are $g$ quaternionic constraints that are  imposed to close these handles: one of the peripheral elements of each of these $(0,3)$ groups  must be conjugate to the inverse of the other peripheral element. Note that, corresponding to each peripheral element there are $7$ natural real parameters: the real trace and two projective points. So, the number of real parameters reduces to $42g-42-7g=35g-42$. But there are $g$ twist-bend parameters $\kappa_i= (s_i,\psi_i,\xi_i, k_{1i}, k_{2i})$, one for each handle,  and each contributes $7$ real parameters. Thus, we need  $35g-30 + 7g=42g-42$ real parameters to specify $\rho$ up to conjugacy.

If two representations have the same coordinates, then the coordinates of the $(0,3)$ groups are the same and so they are  conjugate. Further, it follows from \propref{tb1} and \propref{tw2} that the $(0,4)$ groups and the $(1,1)$ groups are also determined uniquely up to conjugacy while attaching the $(0,3)$ groups. Hence,  representations with the same parameters are conjugate. Conversely, if two representations are  conjugate, then clearly they have the same coordinates.

This proves the theorem. \end{proof}
\section{Hyperbolic Pairs in  $\Sp(1,1)$}\label{fn1}
The group $\Sp(1,1)$ acts by isometries of the quaternionic hyperbolic line $\h^1$. Note that ${\mathrm{PSp}}(1,1)$ is  isomorphic to the isometry group $\PO(4,1)$ of the real hyperbolic $4$ space. In this isomorphism, the group $\Sp(1,1)$ acts on the quaternionic model of $\rh^4$ by linear fractional transformations. This provides a quaternionic analogue of the classical M\"obius  transformations of the Riemann sphere. We refer to the book \cite{pabo} for more details on this action. 

 Here, we view  $\Sp(1,1)$ as a subgroup of $\Sp(2,1)$ that preserves a one-dimensional totally geodesic quaternionic subspace, a copy of $\h^1$,  in $\h^2$. From this viewpoint, we shall follow the framework of the previous sections and will look at the linear action of $\Sp(1,1)$ on $\h^1$.

\subsection{Real Trace} The following is a special case of \cite[Theorem 3.1]{gop}. 
\begin{prop}\label{rt2}
Let $A$ be an element in $\Sp(1,1)$. Let $A_{\C}$ be the
corresponding element in ${\rm GL}(4, \C)$. The characteristic polynomials
of $A_{\C}$ is of the form
\begin{equation*}\chi_A(x)= x^4-a x^3+b x^2 -a x+1 =x^2 g(x+x^{-1}),\end{equation*}
where $a, b \in \R$. Let $\Delta$ be the negative of the discriminant of the polynomial $g_A(t)=g(x+x^{-1})$. Then the conjugacy class of  $A$ is determined by the real numbers $a$ and $b$. Further, $A$ is  loxodromic if and only if $\Delta >0$. 
\end{prop}
This gives us the following definition. 
\begin{definition}
Let $g$ be a hyperbolic isometry of $\h^1$. The real tuple $(a, b)$ in \propref{rt2} is called the  \emph{real trace} of $g$ and shall be denoted by $tr_{\R}(g)$.
\end{definition}

Thus the real trace of a hyperbolic element $g$ of $\Sp(1,1)$ corresponds to  a point on $\R^2$. When $g$ is strictly hyperbolic, then $4b={a^2}+8$. In this case, the real trace is determined by a parameter on $\R$. 

The real traces of loxodromic elements in $\Sp(1,1)$ are given by the following subset of $\R^2$:
$${\rm D}_1=\{(a, b) \in \R^2~ | ~b^2>4a\}.$$

The following result follow similarly as \lemref{lox}. 
\begin{lemma}\label{lel}
Let $A$, $A'$ be hyperbolic elements in $\Sp(1,1)$. Then $A=A'$ if and only if they have the same attracting (or repelling) fixed point,  the same  real trace and the same projective point.
\end{lemma}\label{lox1}
\subsection{Quadruple of boundary points} Next we observe the following analogue of Cao's theorem. Cao proved \thmref{cao} assuming $n \geq 2$. However, the same proof boils down to a much simpler form when $n=1$.  A version of this theorem follows from the work in \cite{gwli}, however there the authors have defined cross ratios using quaternionic arithmetic by the identification of $\Sp(1,1)$ with the quaternionic linear fractional transformations. We can prove the following lemma by using similar ideas in proof of \thmref{cao} of Cao. 

Given an ordered quadruple of pairwise distinct points $(z_1, z_2, z_3, z_4)$ on $\partial \h^2$, their Kor\'anyi-Reimann quaternionic cross ratio is defined by
$$\X(z_1, z_2, z_3, z_4)=[z_1,z_2,z_3,z_4]={\langle {\bf z}_3, {\bf z}_1 \rangle \langle {\bf z}_3, \bf z_2 \rangle}^{-1} { \langle {\bf z}_4, {\bf z}_2\rangle \langle   {\bf z}_4, {\bf z}_1 \rangle^{-1}}.$$

For a pair of hyperbolic elements $(A, B)$ of $\Sp(1,1)$, define 
$$\X(A, B)=\X(a_A, r_A, a_B, r_B).$$ 
\begin{lemma}\label{quad1}
Let $Z=(z_1, z_2, z_3, z_4)$ and $W=(w_1, w_2, w_3, w_4)$ be two quadruple of pairwise distinct points in $\partial \h^1$.  Then there exists an isometry $h \in \Sp(1,1)$ such that $h(z_i)=w_i$, $i=1,2,3,4$, if and only if
$$\Re (\X(z_1, z_2, z_3, z_4))= \Re (\X(w_1, w_2, w_3, w_4)) \hbox{ and }$$
$$ |\X(z_1, z_2, z_3, z_4)|=|\X(w_1, w_2, w_3, w_4)|.$$
\end{lemma}
\begin{proof} Suppose that $\Re (\X(z_1, z_2, z_3, z_4))= \Re (\X(w_1, w_2, w_3, w_4))
, |\X(z_1, z_2, z_3, z_4)|=|\X(w_1, w_2, w_3, w_4)|.$ We want to find $h \in \Sp(1,1)$ such that $h(z_i)=w_i$, $i=1,2,3,4$.
Without loss of generality assume that $z_1=w_1=o$ and $z_2=w_2=\infty.$ By the given hypothesis, $|\X(o, \infty, z_3, z_4)|=|\X(o,\infty, w_3, w_4)|$ implies that $\dfrac{|z_3|}{|w_3|}=\dfrac{|z_4|}{|w_4|}.$ Now, by using the fact that all $z_i,w_i$ has zero real part together with condition $\Re (\X(o, \infty, z_3, z_4))= \Re (\X(o, \infty, w_3, w_4))$ gives equality of the pairs of angles between the vectors $\Im\big(\dfrac{z_3}{|z_3|}\big), ~\Im\big(\dfrac{z_4}{|z_4|}\big)$, and $\Im\big(\dfrac{w_3}{|w_3|}\big),\Im\big(\dfrac{w_4}{|w_4|}\big)$. 
So there exist $\psi \in \Sp(1)$ such that $z_3= t\bar \psi w_3 \psi$ and $z_4= t\bar\psi w_4 \psi$, where  $t=\dfrac{|z_3|}{|w_3|}=\dfrac{|z_4|}{|w_4|}$.  Thus, we get the required isometry 
$$h=\begin{bmatrix} \dfrac{\psi}{\sqrt{t}} & 0 \\ 0 & \sqrt{t}\psi \end{bmatrix},$$ in $\Sp(1,1)$ such that $h(o)=o$, $h(\infty)=\infty$, $h(z_3)=w_3$ and $h(z_4)=w_4$.
\end{proof}
Thus in this case the moduli of ordered quadruple of points up to $\Sp(1,1)$ congruence is determined by the tuple $(\Re(\X(p)), |\X(p)|)$. We denote it by ${\M}(1)$. An explicit description of this space may be obtained from the work of \cite{gwli}. From \cite[Proposition 10]{gwli}, it follows that ${\M}(1)$ is embedded in $\R^4$.

\subsection{Proof of \thmref{plox1}}
\begin{proof}
Suppose $(A, B)$ and $(A', B')$ be pairs of loxodromics such that $tr_{\R}(A)=tr_{\R}(A')$, $tr_{\R}(B)=tr_{\R}(B')$, $\Re (\X(A, B))=\Re (\X(A', B'))$ and $|\X(A, B)|=|\X(A', B')|$.  By \lemref{quad1}, there is an element $C$ in $\Sp(1,1)$ such that $C(a_A)=a_{A'}$, $C(r_A)=r_{A'}$, $C(a_B)=a_{B'}$ and $C(r_B)=r_{B'}$. Therefore $A'$ and $CAC^{-1}$ have the same attracting and the same repelling fixed points. Since they also have the same real trace and the same projective point,   by \lemref{lel}, $CAC^{-1}=A'$. Similarly, $CBC^{-1}=B'$. This completes the proof.
\end{proof}

\medskip  
Given a representation  $\rho$ in  $\dchi_o(\F_2, \Sp(1,1))$, we associate the following tuple to it:
$$(tr_{\R} (\rho(x)), tr_{\R} (\rho(y)), \Re (\X(A, B)), |\X(A, B)|).$$
Let $\mathcal M(1)$ denote the orbit space of ${\M}(1)$ under the natural $S_4$ action on the quadruples. 
This gives the following.
\begin{corollary}
$\dchi_o(\F_2, \Sp(1,1))$  is parametrized by a $\C \P^1 \times \C \P^1$ bundle over the topological space  ${\rm D}_1 \times {\rm D}_1 \times \mathcal M(1)$. \end{corollary}

This implies the following theorem that follows similarly as mentioned in the previous section. A twist-bend $E_K(s, \psi, k)$ in this case would corresponds to four real degrees of freedom given by the real trace $(s, \psi)$ and a projective point $k$. The rest follows similarly as in the previous section. Since the arguments are very similar, we omit the details. 

\begin{theorem}
Let $\Sigma_g$ be a closed surface of genus $g$ with a curve system $\mathcal C=\{\gamma_j
\}$, $j=1, \ldots, 3g-3$. Let $\rho: \pi_1(\Sigma_g) \to \Sp(1,1)$ be a geometric  representation of the surface group $\pi_1(\Sigma_g)$ into $\Sp(1,1)$. Then we need $20g-20$ real parameters to determine $\rho$ in the deformation space $\bchi(\pi_1(\Sigma_g), \Sp(1,1))/\Sp(1,1)$.
\end{theorem}

\subsection{Hyperbolic  pairs in $\GL(2, \H)$}\label{gl}
 The group $\GL(2, \H)$ acts on $\partial \rh^5=\widehat \H=\H \cup \{\infty\}$ by linear fractional transformations:
$$\begin{bmatrix} a & b \\ c & d \end{bmatrix}: z \mapsto (az+b)(cz+d)^{-1},$$
 and this action is extended over the hyperbolic space by Poincar\'e extensions. This action identifies the projective general linear group $\PGL(2, \H)=\GL(2, \H)/Z(\GL(2, \H))$ to the group $\PO(5,1)$,  see \cite{pabo}  or \cite{g9} for more details. Note that $Z(\GL(2, \H))$ is isomorphic to the multiplicative group $\R^{\ast}= {\R \setminus 0}$.

\subsubsection{$3$-Simple Loxodromics}  Let $g$ be a $3$-simple loxodromic element in $\GL(2, \H)$.  Up to conjugacy, it is of the form
$$E_{r, s, \theta, \phi}=\begin{bmatrix} re^{i \theta} & 0 \\ 0 & se^{i \phi}\end{bmatrix},~ r>0, ~s>0,~ \theta, \phi \in (0, \pi).$$
As above, using the embedding of $\GL(2, \H)$ into $\GL(4, \C)$ one can define \emph{real trace} of a $3$-simple  loxodromic in $\GL(2, \H)$.  It follows from the work in \cite{g9} or \cite{ps} that the \emph{real traces} of  $3$-simple loxodromic elements of $\GL(2, \H)$ correspond to  a three real parameter family $(a, b, c)$, $a \neq c$, that forms an open subset of $\R^3$. We denote this open subset by ${\rm D}_3$. This can be seen using ideas similar to the proof of  \lemref{dol}.

\subsubsection{Cross Ratios} As mentioned in the previous case, in this case Gwynne and Libine has defined cross ratio of four boundary points analogous to the cross ratio of four points on the Riemann sphere. Gwynne and Libine have used them to obtain configuration of $\GL(2, \H)$-congruence classes of quadruples of pairwise distinct points on $\widehat \H$. Let  $\mathcal C_4$ denote the space of quadruples of points on $\partial \rh^5$ up to $\GL(2, \H)$ congruence. It follows from the work \cite{gwli} that the subspace  of $\mathcal C_4$ consisting of  quadruples of pairwise distinct points on $\widehat \H$ in general position (i.e. the four points do not belong to a circle),  is a real two-dimensional subspace of $\R^4$. Let $\mathfrak C_4$ denote  the orbit space of $\mathcal C_4$ under the natural $S_4$ action on the tuples.

\subsubsection{The deformation space} Let $\dchi_{*}(\F_2, \GL(2,\H))$ denote the subset of the deformation space $\dchi_o(\F_2, \GL(2, \H))$  consisting of elements $\rho$ such that both $\rho(x)$ and $\rho(y)$ are $3$-simple loxodromics. The following is evident using similar arguments as in the previous sections. 
\begin{theorem}
The set $\dchi_*(\F_2, \GL(2,\H))$  is a $(\C \P^1)^4$ bundle over the  topological space ${\rm D}_3 \times  {\rm D}_3 \times \mathfrak C_4$.
\end{theorem}

Thus we need a total $16$ degrees of freedom to specify an element uniquely in $\dchi_*(\F_2, \GL(2,\H))$. Note that the real dimension of the group $\GL(2, \H)$ is also $16$. However the group $\PGL(2, \H)$ has $15$ real dimension. The group $\GL(2, \H)$ is fibered over $\PGL(2, \H)$ by the punctured real line. This fibration induces a fibration of $\dchi_*(\F_2, \GL(2, \H)$ over $\dchi_*(\F_2, \PGL(2, \H)$ by the punctured real line, and thus we need $15$ real parameters to determine a irreducible point on $\dchi_*(\F_2, \PGL(2, \H))$. 
The following theorem can be proved following similar arguments as earlier. 

\begin{theorem}
Let $\Sigma_g$ be a closed surface of genus $g$ with a curve system $\mathcal C=\{\gamma_j\}$, $j=1, \ldots, 3g-3$. Let $\rho: \pi_1(\Sigma_g) \to \PGL(2, \H)$ be a geometric  representation of the surface group $\pi_1(\Sigma_g)$ into $\PGL(2, \H)$. Then we need $30g-30$ real parameters to determine $\rho$ in the deformation space $\bchi(\pi_1(\Sigma_g), \PGL(2, \H))/\PGL(2, \H)$.
\end{theorem}

\bigskip 
\begin{ack}
We  thank Giannis Platis and Wensheng Cao for many comments and suggestions on a first draft of this paper. We are grateful to the referee for writing an elaborate report on our paper and for many suggestions to improve the structure of the paper.  

\medskip A part of this work was carried out when one of the authors, \hbox{Gongopadhyay},  was visiting the UNSW Sydney supported by the Indo-Australia EMCR Fellowship of the Indian National Science Academy (INSA): thanks to UNSW for hospitality and INSA for the fellowship during the visit. Thanks are also due to Anne Thomas for organising a seminar by Gongopadhyay on this topic at the Sydney University. 

\medskip Kalane thanks a UGC research fellowship for supporting him through out this project.  Gongopadhyay acknowledges partial support from SERB-DST MATRICS project: MTR/2017/000355.

\end{ack}


\end{document}